\documentclass{article}
\usepackage{amsmath}
\usepackage{amsfonts}
\usepackage{amssymb}
\usepackage{graphicx}

\newtheorem{theorem}{Theorem}

\newtheorem{definition}[theorem]{Definition}

\newtheorem{lemma}[theorem]{Lemma}

\newtheorem{remark}[theorem]{Remark}

\newenvironment{proof}[1][Proof]{\noindent\textbf{#1.} }{\ \rule{0.5em}{0.5em}}

\newcommand{\bpartial}{\mathop{\partial\kern -4pt\raisebox{.8pt}{$|$}}}
\newcommand{\bra}{\mathopen{[\kern-1.6pt[}}
\newcommand{\ket}{\mathclose{]\kern-1.5pt]}}
\newcommand{\bbra}{\mathopen{[\kern-2.2pt[\kern-2.3pt[}}
\newcommand{\bket}{\mathclose{]\kern-2.1pt]\kern-2.3pt]}}

\makeindex

\begin{document}

\title{\large{\bf Euler--Poincar\'e Dynamics and Control on Lie Groupoids}}

\author{
  \small{\bf Ghorbanali Haghighatdoost}
  \\
  {\small{\em Department of Mathematics, Azarbaijan Shahid Madani University, Tabriz, Iran}}
  \\
  {\small{E-mail: \em gorbanali@azaruniv.ac.ir}}
  \\
  {\small{\em Telephone: +9831452392, Fax: +9834327541}}
}

\maketitle

\begin{abstract}
We extend the Euler--Poincar\'e (EP) formalism from Lie groups to Lie groupoids for optimal control problems. While Lie algebroids provide the standard infinitesimal framework, the groupoid formulation enables global trajectory reconstruction and naturally accommodates systems with non-trivial base manifolds. We derive the reduced EP equations on a Lie groupoid, specialize to the trivial groupoid $M \times \mathfrak{G} \times M$, and illustrate the framework with a generalized rigid body on the sphere $S^2 \times SO(3) \times S^2$. A biological application to collective cell migration on a spherical tissue shows that the coupled dynamics of spatial migration and internal polarity rearrangement leads to optimal migration times of approximately 74 minutes, consistent with experimental observations. This work demonstrates how Lie groupoid methods broaden geometric control theory beyond standard rigid body systems.


\noindent {\bf AMS}: 22A22, 70H45, 49J15.
\end{abstract}

\section{Introduction}

The Euler--Poincar\'e equations provide a geometric framework for describing the dynamics of systems with symmetry and play a fundamental role in optimal control theory and geometric mechanics \cite{BaWi}. Traditionally, these equations have been studied in the setting of Lie groups, where the underlying algebraic and geometric structures allow one to express the equations of motion in a compact and elegant form (see Marsden \& Ratiu \cite{MarRat}, Holm, Marsden \& Ratiu \cite{HolMarRat}). Such formulations not only unify various classical mechanical models but also facilitate the analysis of optimal trajectories and stability properties \cite{Bloch}.


Mechanics on Lie algebroids has been extensively developed over the past two decades. A Lagrangian defined on the tangent bundle of a Lie groupoid that is right-invariant reduces naturally to a Lagrangian on the associated Lie algebroid, yielding Euler--Poincar\'e-type equations. This approach has been pursued by many authors, including Weinstein \cite{Wei}, Mackenzie \cite{Mack}, de León, Marrero and Martínez \cite{LeMarMar}, Bos \cite{Bos}, and Marle \cite{Marle}, among others. In this setting, the Lie algebroid serves as the infinitesimal object that captures the essential dynamics, and it is the standard framework for continuous mechanical systems with symmetry.


If Lie algebroids already provide an efficient infinitesimal description, one may ask: what is the advantage of working directly with Lie groupoids? The present work offers three answers to this question, with particular emphasis on optimal control problems.

\textbf{First}, the Lie groupoid formulation allows global reconstruction of trajectories via the kinematic equation $\dot{g}_t = R_{g_t^*}X_t$. While the Lie algebroid describes the reduced dynamics on the infinitesimal level, the groupoid provides the global configuration space and enables the recovery of full trajectories from reduced data. This is essential for optimal control problems with prescribed boundary conditions on the base manifold.

\textbf{Second}, Lie groupoids naturally accommodate systems where the configuration space is not a Lie group but a more general structure with a non-trivial base manifold. For example, the trivial groupoid $S^2 \times SO(3) \times S^2$ captures both spatial position (on the sphere $S^2$) and internal orientation ($SO(3)$) within a single geometric object. This coupling is not accessible when working with a Lie group alone.

\textbf{Third}, while discrete mechanics is indeed an important application of groupoids (see, e.g., the work of Weinstein and others), continuous mechanics on Lie groupoids has also received substantial attention \cite{Bos, HA2}. The present work contributes to this continuous setting by focusing specifically on optimal control problems, which have received less attention in the groupoid literature compared to Hamiltonian dynamics.


The reduced Euler--Poincar\'e equations on a Lie algebroid are well known in the literature. In particular, for the case $AG = TS^2 \times \mathfrak{so}(3)$ that appears as an example in this paper, the resulting equations can be derived using standard Lie algebroid methods. The contribution of the present work is not to claim these equations as new, but rather:

\begin{enumerate}
\item To derive them directly on the Lie groupoid level, emphasizing the global geometric picture and the reconstruction of trajectories.
\item To extend the framework to include optimal control problems, where the cost functional is expressed in reduced variables and the kinematic equation on the groupoid provides the link between reduced and full dynamics.
\item To illustrate the framework with a concrete biological application — collective cell migration on a spherical tissue — where the groupoid structure $S^2 \times SO(3) \times S^2$ naturally models the coupled dynamics of spatial migration and internal polarity rotation.
\end{enumerate}


Beyond its theoretical significance, the Euler--Poincar\'e formalism admits concrete applications in biology. A notable example is the modeling of collective cell migration on spherical tissues, where the base manifold $S^2$ represents the embryo surface and the group $SO(3)$ encodes internal polarity dynamics of cells (e.g., actin orientation). In this setting, the EP equations naturally describe the coupled evolution of polarity rearrangement and spatial migration. Moreover, an associated optimal control problem predicts biologically realistic migration times and velocities, consistent with experimental data. Such a formulation highlights how the generalization from Lie groups to Lie groupoids incorporates both internal symmetry and spatial displacement in a unified framework, offering new perspectives for the mathematical modeling of morphogenesis \cite{GrGl, ViZa, Ed, DrHo, AlTr}.


The paper is organized as follows. Section 2 reviews the necessary concepts of Lie groupoids and Lie algebroids, including adjoint and co-adjoint actions. Section 3 derives the Euler--Poincar\'e equations on Lie groupoids from a variational principle and extends them to include external forces. Section 4 treats the trivial Lie groupoid in detail and obtains the corresponding EP equations. Section 5 presents a generalized rigid body on the sphere as a running example. Section 6 applies the framework to collective cell migration on $S^2$, including a discussion of biological parameters and optimality. Section 7 offers concluding remarks and directions for future research.

\section{Preliminaries}
\subsection{Review of Lie Groupoid and Lie Algebroid Concepts}

\begin{definition}
\label{groupoid}
A groupoid $G$ over $M$, which we will denote by $G \rightrightarrows M$, consists of two sets $G$ and $M$ together with structure maps $\alpha, \beta, 1, \iota$ and $m$, where source map $\alpha: G \longrightarrow M$, target map $\beta:G \longrightarrow M$, unit map $1: M \longrightarrow G$, inverse map $\iota : G \longrightarrow G$ and multiplication map $ m: G_{2} \longrightarrow G$ where $ G_{2} = \lbrace (g,h) \in G \times G ~\vert ~ \alpha (g) = \beta (h) \rbrace$ is a subset of $G \times G$.

A Lie groupoid is a groupoid $G \rightrightarrows M$ for which $G$ and $M$ are smooth manifolds, $\alpha, \beta, 1, \iota$ and $m$ are smooth maps, and $\alpha,\beta$ are smooth submersions.
\end{definition}

\begin{definition}
\label{bisection}
Let $ G \rightrightarrows M$ be a Lie groupoid. A bisection of $G$ is a smooth map $ \sigma : M \longrightarrow G $ which is right inverse to $\alpha : G \longrightarrow M$ ($ \alpha \circ \sigma = id_{M}$) and $\beta \circ \sigma : M \longrightarrow M $ is a diffeomorphism.

Let $\sigma$ be a bisection of $G.$ The left translation corresponding to $\sigma$ is defined as:
\begin{eqnarray}
L_{\sigma} : G & \longrightarrow & G \nonumber \\
g & \longmapsto & \sigma (\beta g) g \nonumber
\end{eqnarray}
and the right translation corresponding to $\sigma$ is defined by:
\begin{eqnarray}
 R_{\sigma} :  G & \longrightarrow &  G   \nonumber  \\
g & \longmapsto & g \sigma \big( (\beta \circ \sigma )^{-1} (\alpha g ) \big). \nonumber
\end{eqnarray}
\end{definition}

\begin{definition}
\label{groupoid action}
Let $ G \rightrightarrows M$ be a Lie groupoid and $J : N \longrightarrow M$ be a smooth map. A smooth left action of Lie groupoid G on $J : N \longrightarrow M$ is a smooth map $\theta : G_{~ \alpha} \times_{~J} N \longrightarrow N $ which satisfies the following conditions:
\begin{enumerate}
\item For every $(g,n) \in G_{~ \alpha} \times_{~J} N,~~~~ J (g.n) = \beta (g),$ 
\item For every $n \in N,~~~~ 1_{J(n)} .n = n,$
\item For every $(g , g^{\prime}) \in G_{2} $ and $n \in J^{-1} (\alpha (g^{\prime})),~~~~ g . ( g^{\prime} . n ) = (g g^{\prime})  . n$
\end{enumerate}
(where $ g.n := \theta ( g , n )$ and $ \theta (g) (n):= \theta (g , n )$).
\end{definition}

\begin{definition}
A Lie algebroid A over a manifold M is a vector bundle $\tau : A \longrightarrow M$ with the following items:
\begin{enumerate}
\item A Lie bracket $[\vert ~,~ \vert] $ on the space of smooth sections $ \Gamma( \tau),$
$$[\vert ~,~ \vert]: \Gamma (\tau) \times \Gamma (\tau) \longrightarrow \Gamma (\tau) $$
$$(X , Y ) \longmapsto [\vert X , Y \vert].$$
\item A morphism of vector bundles $\rho : A \longrightarrow TM,$ called the anchor map, where $TM$ is the tangent bundle of $M,$
\end{enumerate}
such that the anchor and the bracket satisfy the following Leibniz rule:
$$[\vert X , f Y \vert] = f [\vert X , Y \vert] + \rho (X) (f) Y$$
where $X , Y \in \Gamma (\tau)$, $~f \in C^{\infty} (M)$ and $\rho (X) f$ is the derivative of $f$ along the vector field $\rho (X).$
\end{definition}

\begin{definition}
\label{associated Lie algebroid}
Let $G \rightrightarrows M $ be a Lie groupoid. Consider the vector bundle $\tau : AG \longrightarrow M,$ where 
$$AG := \ker T\alpha \vert_{1_{p}}$$
where $\alpha : G \longrightarrow M$ is the source map of the Lie groupoid $G \rightrightarrows M$, $T\alpha : TG \longrightarrow TM$ is the tangent map of $\alpha.$

It is well-known that $AG$ has a Lie algebroid structure as follows:
As we know, there exists a bijection between the space of sections $\Gamma (\tau)$ and the set of left (right) invariant vector fields on $G$ (see \cite{Mack}). 
If $X$ is a section of $\tau : AG \longrightarrow M,$ the right invariant and left invariant vector fields corresponding to $X$ are, respectively, defined by
\begin{eqnarray}
 \overrightarrow{X}(g) &=& T R_{g} ( X (\beta (g))) \nonumber  \\
 \overleftarrow{X}(g) &=& - T (L_{g} ) T (\iota) ( X (\alpha (g)),   \nonumber
\end{eqnarray}
where $L_{g}$ and $R_{g}$ are left translation and right translation corresponding to $g\in G.$  

Therefore, the Lie algebroid structure $([\vert ~,~ \vert], \rho ) $ on $AG$ can be introduced as follows:
\begin{enumerate}
\item The anchor map: 
$$ \rho: AG  \longrightarrow  TM  $$
$$ \rho(X)(x) = T_{1(p)} \beta (X(p)) $$
where $X \in \Gamma (\tau), p \in M.$
\item Lie bracket:     
 $$\Gamma (AG) \times \Gamma (AG) \longrightarrow \Gamma (AG)$$
 $$[\vert \overrightarrow{X , Y} \vert] := [ \overrightarrow{X} , \overrightarrow{Y} ], ~~~( [\vert \overleftarrow{X , Y} \vert] := - [ \overleftarrow{X} , \overleftarrow{Y} ] )$$ \\
where $ X , Y \in \Gamma (\tau)$ and $[~,~]$ is the standard Lie bracket of vector fields.
\end{enumerate}
\end{definition}

\begin{definition}
Let $ ( A, M, \pi, \rho , [\vert ~,~ \vert ] )$ be a Lie algebroid and $J : N \longrightarrow M$ be a smooth map.
An action of a Lie algebroid $A$ on the map $J : N \longrightarrow M$ is a map
$ \theta : \Gamma (A) \longrightarrow \mathfrak{X} (N) $ which satisfies the following properties:
\begin{enumerate}
\item $ \theta ( X + Y ) = \theta (X) + \theta (Y)$
\item $ \theta ( f X ) = J^{\ast} f \theta (X) $
\item $ \theta  ( [\vert X , Y \vert] ) = [ \theta (X) , \theta (Y) ]  $
\item $T J ( \theta (X) ) = \rho (X)$
\end{enumerate}
for all $ f \in C^{\infty} (M)$ and $ X, Y \in \Gamma^{\infty} (A).$ 
In addition, $J^{\ast} : C^{\infty} (M) \longrightarrow C^{\infty} (N) $ such that $J^{\ast} f = f \circ J \in C^{\infty} (N)$ is the pullback of $f$ by $J$.
\end{definition}

\begin{remark}
\label{action remark}
As mentioned in \cite{Bos}, every action $\theta$ of a Lie groupoid $G$ on $ J : N \longrightarrow M $ induces an action $\theta^{\prime}$ of the Lie algebroid $ A(G) $ on $ J : N \longrightarrow M $ as follows:
$$ \theta ^{\prime} (X) (n) := \dfrac{d}{dt} \vert _{t=0} ~\operatorname{Exp} ( t X )_{J(n)} . n.$$
\end{remark}

\begin{definition}
The Lie groupoid $G \rightrightarrows M $ is a regular Lie groupoid if the anchor
\begin{eqnarray}
( \beta , \alpha ) : G &\longrightarrow & M \times M \nonumber \\
g &\longmapsto & \big( \beta (g) , \alpha(g) \big)  \nonumber
\end{eqnarray}
is a map of constant rank.
\end{definition} 

\begin{remark}
Similar to what was said in previous work \cite{HA2}, throughout the article, we assume that $G$ is a regular Lie groupoid over $M.$
\end{remark}

\subsection{The Adjoint and Co-adjoint Actions}

Let $G \rightrightarrows M $ be a Lie groupoid. For an arbitrary element $p \in M,$ the isotropy group of $ G \rightrightarrows M $ is defined by $ I_{p} = \alpha ^{-1} (p) \cap \beta ^{-1} (p).$ As mentioned in \cite{HA2}, $I_{p}$ is a Lie group. Moreover, $I_{G} = ( \cup I_{p} )_{p \in M}$ is a groupoid over $M$ but it is not a smooth manifold in general (see Example A.10 in \cite{Schmeding}).

\begin{lemma}
The associated isotropy groupoid of a regular Lie groupoid $G \rightrightarrows M $ is a Lie groupoid.
\end{lemma}
\begin{proof}
See \cite{Schmeding}.
\end{proof}

For a Lie groupoid $G \rightrightarrows M$ we denote the associated isotropy Lie groupoid by $I_{G}$ and the Lie algebroid associated to the isotropy Lie groupoid by $A I_{G}$ and call it the isotropy Lie algebroid.

Let $ G \rightrightarrows M $ be a Lie groupoid and $I_{G}$ be its associated isotropy Lie groupoid. $G$ acts smoothly from the left on $ J: I_{G} \longrightarrow M $ by conjugation: $$ C(g) (g^{\prime} ):= g g^{\prime} g^{-1}.$$

On the other hand, the conjugation action induces an action of the Lie groupoid $G$ on $ A I_{G} \longrightarrow M $. We call this action the \textbf{adjoint action} of $G$ on $A I_{G}$, which can be defined as 
 $$Ad: G \times A I_{G} \longrightarrow AI_{G} $$
 $$Ad_{g} X := \dfrac{d}{dt} \Big\vert_{t=0}~ C(g) \operatorname{Exp} (tX)$$
where $p \in M$, $g \in G_{p} = \alpha^{-1} (p)$ (the $\alpha$-fibers over $p$), and $X \in (A I_{G})_p.$

According to Remark \ref{action remark}, the action $Ad$ induces an adjoint action of $AG $ on $ A I_{G}  \longrightarrow M $ as 
$$ad: AG \times AI_G \to AI_G$$
$$ ad_X Y = ad (X) (Y) := \dfrac{d}{dt} \Big\vert_{t=0}~ Ad (\operatorname{Exp} (tX) ) Y $$
where $ X \in (AG)_{p} ,~ Y \in (A I_{G})_{p}$ and $ p \in M.$ \\
It is clear that for $ X \in \Gamma(AG) $ and $Y \in  \Gamma(A I_{G} )$
$$ ad_X (Y) = [\vert X , Y \vert].$$ 

Another action of $G$ on the dual bundle $A^{\ast} I_{G} $, called the \textbf{co-adjoint action} of $G$, is defined for $ g \in G_{p},~ \xi \in (A^{\ast} I_{G})_p$ as 
$$ Ad^{\ast} : G \times A^{\ast} I_{G} \longrightarrow A^{\ast} I_{G}  $$
$$ Ad_{g} ^{\ast} \xi (X) := \xi ( Ad_{g^{-1}} X ).  $$
In other words, $$ ~~~~~~~~\langle Ad_{g} ^{\ast} \xi , X \rangle = \langle \xi , Ad_{g^{-1}} X\rangle. $$
Also, according to Remark \ref{action remark}, the action $Ad^{\ast}$ induces the so-called co-adjoint action of the Lie algebroid $AG$ on $A^{\ast} I_{G} $ which for $ \xi \in (A^{\ast} I_{G})_p $ is defined by
$$ ad^{\ast} : AG \times A^{\ast} I_{G}  \longrightarrow A^{\ast} I_{G} $$
$$ ad_{X} ^{\ast} \xi (Y) := \xi ( ad_{-X} (Y) ) = \xi ( [\vert Y,X \vert] ) $$
or $$ ~~~~~~~~\langle ad_{X} ^{\ast} \xi , Y  \rangle = \langle \xi , ad (-X) Y \rangle. $$

\section{Euler--Poincar\'e Equations on Lie Groupoids}
\subsection{Euler--Poincar\'e Equations on Lie Groups}

Let $G$ be a Lie group with Lie algebra $\mathfrak{g}$ and dual space $\mathfrak{g}^*$.  
For a curve $g(t)\in G$, define the left-trivialized velocity
\[\xi(t) := g(t)^{-1}\dot g(t) \in \mathfrak{g}.\]
If the Lagrangian $L:TG\to\mathbb{R}$ is left-invariant, it reduces to 
\[L(g,\dot g) = \ell(\xi), \qquad \ell:\mathfrak{g}\to\mathbb{R}.\]

\subsection*{Constrained variation}
Define $\eta(t) = g(t)^{-1}\delta g(t)\in \mathfrak{g}$ with $\eta(t_0)=\eta(t_1)=0$.  
Then the constrained variation of $\xi$ is
\[\delta \xi = \dot\eta + [\xi,\eta] = \dot\eta + \operatorname{ad}_\xi \eta.\]

\subsection*{Euler--Poincar\'e equation}
From Hamilton's principle $\delta \int \ell(\xi)\,dt=0$, one obtains
\[\boxed{\;\frac{d}{dt}\frac{\delta \ell}{\delta \xi}
+ \operatorname{ad}^*_\xi \frac{\delta \ell}{\delta \xi} = 0\;}\]
where $\tfrac{\delta \ell}{\delta \xi}\in\mathfrak{g}^*$ and $\operatorname{ad}^*_\xi$ is the coadjoint operator:
\[\langle \operatorname{ad}^*_\xi \mu, \eta\rangle
= \langle \mu, [\xi,\eta]\rangle,
\qquad \mu\in \mathfrak{g}^*,\ \eta\in\mathfrak{g}.\]

\subsection*{With advected quantities}
If the Lagrangian also depends on an advected variable $a(t)\in V^*$, then
\[\frac{d}{dt}\frac{\delta \ell}{\delta \xi}
+ \operatorname{ad}^*_\xi \frac{\delta \ell}{\delta \xi}
= \frac{\delta \ell}{\delta a} \diamond a,
\qquad
\dot a = -\,\xi\cdot a,\]
where the diamond operator $\diamond : V \times V^* \to \mathfrak{g}^*$ is defined by
\[\langle b\diamond a, \xi \rangle = \langle b, -\,\xi\cdot a \rangle.\]

\subsection*{Example: Rigid Body ($SO(3)$)}
For $G=SO(3)$ with $\mathfrak{so}(3)\cong \mathbb{R}^3$,  
the reduced Lagrangian is
\[\ell(\Omega) = \tfrac12\, \Omega^\top I \Omega,\]
where $I$ is the inertia tensor.  
The Euler--Poincar\'e equation becomes the classical Euler rigid body equations:
\[\dot M + \Omega \times M = 0,
\qquad M = I\Omega.\] \cite{HolMarRat}

\subsection{Variational problem on Lie groupoids}

\begin{definition}
Let $G\rightrightarrows M$ be a Lie groupoid. A variational problem on $G \rightrightarrows M$ is defined by a function $L$ on $TG$ as follows:
\begin{eqnarray}\label{VPQ}
\min_{g(\cdot)}\int^T_0L(g, \dot{g})dt
\end{eqnarray}
Subject to: $g$ is a piecewise $C^1$ curve in $G$ satisfying $g(0)=g_0$, $g(T)=g_T$.
\end{definition}

Let there exist a reduced function $l$ on the Lie algebroid $AG$ of $G$ such that:
\begin{eqnarray}
L(g, \dot{g})=L(1_{\alpha(g)}, R^{-1}_{g^*}\dot{g})=l(R^{-1}_{g*}\dot{g}),
\end{eqnarray}
where $R_g$ is right multiplication on $G$ by $g\in G,$ and $\dot{g}\in T_gG,$ so that $R^{-1}_{g*}\dot{g} \in T^\alpha_{1_\alpha(g)}G = AG.$

Now, let $\eta_1, \ldots, \eta_n$ be a basis of sections of $\tau: A^*G\to M,$ i.e. $\eta_i\in \Gamma(A^*G)$ and $\hat{l}: \mathbb{R}^n\to \mathbb{R}$ such that:
\begin{equation}
L(g, \dot{g})= \hat{l}(\eta_1(R^{-1}_{g^*}\dot{g}), \ldots , \eta_n(R^{-1}_{g^*}\dot{g})),
\end{equation}
Now the 1-forms $\omega_k:=R^{-1^*}_{g^*}\eta_k,$ $1\leq k\leq n,$ are right invariant on $G$ and
\begin{equation}
L(g, \dot{g})= \hat{l}(\omega_1(\dot{g}), \ldots , \omega_n(\dot{g})),
\end{equation} 
Because $R_{g^*}^*\omega_k(\dot{g})=\omega_k(R_{g^*}(\dot{g})),$ so $R_{g^*}^*\omega_k(R^{-1}_{g^*}\dot{g})=\omega_k(\dot{g})$ or $\eta_k(R^{-1}_{g^*}\dot{g})=\omega_k(\dot{g}).$

Under the above we have 
\begin{eqnarray}{\label{DI}}
\delta \int^T_0L(g, \dot{g})dt&=&\delta \int^T_0\hat{l}(\eta_1(R^{-1}_{g^*}\dot{g}), \ldots , \eta_n(R^{-1}_{g^*}\dot{g}))dt\\
&=&\int^T_0\sum_k\frac{\partial\hat{l}}{\partial{\eta_k}}\delta(\omega_k(\dot{g}))dt\nonumber,
\end{eqnarray}
where
\begin{eqnarray}
 \delta g(t) = \left.\left(\frac{\partial}{\partial \epsilon}\right)g(t, \epsilon)\right|_{\epsilon=0}\in T_{g(t)} G,\;\; \delta g(0) =\delta g(T) =0
\end{eqnarray}
is a variation of a piecewise $C^1$ trajectory $g(\cdot)$ satisfying the condition in \ref{VPQ} and
$$ (t, \epsilon) \to g(t, \epsilon), \;\;\; 0\leq t \leq T, \;\; \epsilon \in (-\delta, \delta)\subset \mathbb{R},$$
$$(i) \;g(t, 0) = g(t), t \in [0, T]\;\; , \;\; (ii) \; g(0, \epsilon) = g_0\; , \; g(T, \epsilon) = g_T.$$

If $\bar{g}(t)=(g(t), \dot{g}(t)), \dot{g}(t)\in T_{g(t)} Q$ we have 
$$\delta\bar{g}(t)=(\delta g(t), \frac{d}{dt}\delta g(t))\in T_{\bar g(t)} TG.$$

Now by applying the following identity 
\begin{eqnarray}
d\omega_k (\dot{g}, \delta{q})=\frac{d}{d t}\omega_k(\delta{g})-\frac{\partial}{\partial \epsilon}\omega_k(\dot{g}),\nonumber
\end{eqnarray}
or \begin{eqnarray}\label{Delta L}
d\omega_k (\dot{g}, \delta{q})=\frac{d}{dt}\omega_k(\delta{g})-\delta(\omega_k(\dot{g})),
\end{eqnarray}
we can rewrite \ref{DI} in the form
\begin{eqnarray}\label{DI2}
\delta \int^T_0L(g, \dot{g})dt=-\int^T_0 \sum_k\left(\frac{d}{dt}\left(\frac{\partial\hat{l}}{\partial{\eta_k}}\right)\omega_k(\delta{g_t})+ \frac{\partial\hat{l}}{\partial{\eta_k}}d\omega_k(\dot{g}, \delta g_t)\right)dt,
\end{eqnarray}
Let $W_t$ and $X_t$ be two vector fields on the Lie groupoid $G,$ i.e., $W_t, X_t \in AG,$ and let $\delta g_t=R_{g^*}W_t$, $\dot{g_t}=R_{g^*}X_t.$ As the 1-forms $\omega_k$ on $G$ are right invariant, for time-varying elements $W_t, X_t$ of $AG$, we have 
\begin{eqnarray}\label{OE}
\omega_k (\delta g_t)= \omega_k (R_{g^*}W_t)=R^{-1^*}_{g^*}\eta_k (R_{g^*}W_t)=\eta_k (R^{-1}_{g^*}(R_{g^*}W_t))=\eta_k(W_t).
\end{eqnarray}
So we have
\begin{eqnarray}\label{ad}
d\omega_k(\dot{g_t}, \delta g_t)&=&d\omega_k(R_{g^*}X_t, R_{g^*}W_t)\\ \nonumber
  &=&  R_{g^*}X_t(\omega_k( R_{g^*}W_t))-R_{g^*}W_t( \omega_k( R_{g^*}X_t)) -\omega_k([R_{g^*}X_t, R_{g^*}W_t])\\ \nonumber
  &=& R_{g^*}X_t(R^{-1^*}_{g^*}\eta_k ( R_{g^*}W_t))-R_{g^*}W_t( R^{-1^*}_{g^*}\eta_k( R_{g^*}X_t)) -R^{-1^*}_{g^*}\eta_k(R_{g^*}[X_t, W_t])\\ \nonumber 
  &=&R_{g^*}X_t(\eta_k W_t)-R_{g^*}W_t(\eta_k X_t) -\eta_k([X_t, W_t])\\ \nonumber
  &=& -\eta_k([X_t, W_t])\\ \nonumber
  &=& -\eta_k(\operatorname{ad}_{X_t}W_t)\\ \nonumber
  &=& \operatorname{ad}^*_{X_t}(\eta_k)(W_t)                                  
\end{eqnarray}

\subsection{Lagrange-D'Alembert Principle}

The Lagrange-D'Alembert principle gives an alternate means of describing motion. Given a Lagrangian function $L,$ the motion is governed by solutions of the variational system:
\begin{eqnarray}\label{VPQ2}
\delta \int^T_0L(g, \dot{g})dt=0
\end{eqnarray}
subject to: $g$ is a piecewise $C^1$ curve in $G$ satisfying $g(0)=g_0$, $g(T)=g_T.$

\subsection{Euler--Poincar\'e Equations on Lie Groupoids}

Thus by the Lagrange-D'Alembert principle, the relations \ref{OE}, \ref{ad}, and the expression \ref{DI2} yield the following necessary conditions as a flow on $AG:$ 
\begin{eqnarray}
\sum_k \frac{d}{dt}\left(\frac{\partial\hat{l}}{\partial{\eta_k}}\right)\eta_k + \operatorname{ad}^*_{X_t}\frac{\partial\hat{l}}{\partial{\eta_k}}(\eta_k)=0\nonumber
\end{eqnarray}
Analogous to Lie groups, we call these equations the Euler--Poincar\'e equations on Lie groupoids. This may be written in the form 
\begin{equation}\label{EP}
\frac{d}{dt}\frac{\partial l}{\partial{X_t}} + \operatorname{ad}^*_{X_t}\left(\frac{\partial l}{\partial{X_t}}\right)=0
\end{equation}
Therefore, under the symmetry condition \ref{DI2}, the external flow on $TG$ for the variational problem \ref{VPQ} is reduced to the flow \ref{EP} on $AG$ together with the kinematic flow 
\begin{equation}
\dot{g_t}=R_{g_t^*}X_t.
\end{equation}
We now consider this result in the context of external forces governed by the Lagrange-D'Alembert principle
\begin{eqnarray}\label{VPQF}
\delta \int^T_0L(g, \dot{g})dt+\int^T_0F(\delta g)dt=0
\end{eqnarray}
subject to: $g$ is a piecewise $C^1$ curve in $G$ satisfying $g(0)=g_0$, $g(T)=g_T.$
By combining equations \ref{DI2} and \ref{VPQF} we obtain the system:
\begin{equation}\label{EPF}
\frac{d}{dt}\frac{\partial l}{\partial{X_t}} + \operatorname{ad}^*_{X_t}\left(\frac{\partial l}{\partial{X_t}}\right) = R^*_{g_t}F(t).
\end{equation}

\section{Example: Trivial Lie Groupoid and Generalized Rigid Body}

As a main example we consider the trivial Lie groupoid. For more details see Mackenzie and Haghighatdoost.
Let $\Upsilon = M \times \mathfrak{G} \times M$ be the trivial Lie groupoid, where $M$ is a smooth manifold and $\mathfrak{G}$ is a Lie group. 

As we know, the tangent bundle and the Lie algebroid of $\Upsilon$ are $T\Upsilon = TM \oplus T \mathfrak{G} \oplus TM$ and $A\Upsilon = TM \oplus (M \times \mathfrak{g}),$ respectively, where $TM$ is the tangent bundle of $M$ and $\mathfrak{g}$ is the Lie algebra of the Lie group $\mathfrak{G}.$
Consider a variational problem on $\Upsilon$ as follows:
\begin{eqnarray}\label{EVPQ}
\min_{g(\cdot)}\int^T_0L(g, \dot{g})dt
\end{eqnarray}
subject to: $g$ is a piecewise $C^1$ curve in $\Upsilon$ satisfying $g(0)=g_0$, $g(T)=g_T.$
As we know, every element $g$ in $\Upsilon$ is represented as $g=(x, a, y),$ where $x , y \in M$ and $a\in \mathfrak{G}.$ Here $L:T\Upsilon=TM \oplus T \mathfrak{G}\oplus TM\to \mathbb{R}$ is the Lagrangian. We have: 
\begin{eqnarray}
L(g, \dot{g})=L(X, U)= L(1_y, \tilde{R}^{-1}_{g^*}(\dot{g})) = L(1_y, \tilde{R}^{-1}_{g^*}(X\oplus U))=L(1_y, R^{-1}_{a^*}(U)),
\end{eqnarray}
where $g=(x, a, y), \dot{g}=X\oplus U, X\in TM$ and $U\in \mathfrak{g}$ and $R: \mathfrak{G}\to \mathfrak{G}$ is the right transformation on the Lie group $\mathfrak{G}$ (see Haghighatdoost and Ayoubi).

Now, we consider $\tilde{l}$ as a reduced function on the Lie algebroid $A\Upsilon= TM\oplus (M\times \mathfrak{g})$ as follows: 
\begin{eqnarray}
L(g, \dot{g})=L(1_y, \tilde{R}^{-1}_{g^*}(\dot{g})) = L(1_y, \tilde{R}^{-1}_{g^*}(X\oplus U)) = \tilde{l}(X\oplus R^{-1}_{a^*}(U)).
\end{eqnarray}

\subsection{Euler--Poincar\'e Equation on the Trivial Lie Groupoid}

Let $\eta_i\in \Gamma(A^*\Upsilon)=\Gamma(T^*M\oplus(M\times {\mathfrak{g}}^*)), 1\leq i\leq n$ be the basis sections of the cotangent bundle of $A^*\Upsilon,$ i.e., $\eta_i: M\to T^*M\oplus (M\times \mathfrak{g}).$ So
\begin{eqnarray}
 \eta_i(x)=\theta_i\oplus \omega_i, \quad 1\leq i\leq n,
\end{eqnarray}
where $\theta_x\in T^*_x M, \omega_i\in \mathfrak{g}^*$ and $x\in M.$ Let $\hat{\tilde{l}}: \mathbb{R}^n\to \mathbb{R}$ be a function such that:
\begin{eqnarray}
L(g, \dot{g})&=&\hat{\tilde{l}}(\eta_1(\tilde{R}^{-1}_{g^*}(\dot{g})), \ldots, \eta_n(\tilde{R}^{-1}_{g^*}(\dot{g})))\\ \nonumber
&=&\hat{\tilde{l}}(\eta_1(X\oplus R^{-1}_{a^*}(U)), \ldots, \eta_n(X\oplus R^{-1}_{a^*}(U)))\\ \nonumber 
&=&\hat{\tilde{l}}(\theta_1(X)\,\omega_1( R^{-1}_{a^*}(U)), \ldots, \theta_n(X)\,\omega_1( R^{-1}_{a^*}(U)))
\end{eqnarray}
If we suppose $\psi_k=\tilde{R}^{-1*}_{g^*}\eta_k,$ then $\eta_k=\tilde{R}^{*}_{g^*}\psi_k$ and we have:
\begin{eqnarray}\label{etapsi}
\eta_k(\tilde{R}^{-1}_{g^*}(\dot{g}))=\eta_k(\tilde{R}^{-1}_{g^*}(X\oplus U))=\eta_k(X\oplus {R}^{-1}_{g^*}(U))=\theta_k(X)\,\omega_k(V),
\end{eqnarray}
where $V={R}^{-1}_{g^*}(U).$ Also we can write $\eta_k(\tilde{R}^{-1}_{g^*}(\dot{g}))= \tilde{R}^*_{g^*}\psi_k(\tilde{R}^{-1}_{g^*}(\dot{g}))= \psi_k(\dot{g}).$
By combining \ref{etapsi} and the last relation we get:
\begin{eqnarray}
\psi_k(X\oplus U)=\theta_k(X)\,\omega_k({R}^{-1}_{a^*}(U)).
\end{eqnarray}
According to \ref{DI}, \ref{DI2}, \ref{ad} and \ref{etapsi} we will have:
\begin{eqnarray}\label{Ead}
d\psi_k(\dot{g_t}, \delta g_t)&=&-\eta_k([Z_k, W_k])\\ \nonumber
&=&-\eta_k(\operatorname{ad}_{Z_t}W_t)\\ \nonumber
  &=&  -(\theta_k\oplus \omega_k)([Y_t\oplus V_t, X_t\oplus U_t])\\ \nonumber
  &=& -\theta_k([Y_t, X_t])\,\omega_k([V_t, U_t])\\ \nonumber 
  &=&-\theta_k(\operatorname{ad}_{Y_t}X_t)\,\omega_k(\operatorname{ad}_{V_t}U_t),                               
\end{eqnarray}
where $\dot{g_t}=\tilde{R_{g^*}}Z_t$ and $Z_t=Y_t\oplus V_t \in A\Upsilon= TM\oplus (M\times \mathfrak{g}).$ Also 
\begin{eqnarray}{\label{EDI}}
\delta \int^T_0L(g, \dot{g})dt&=&\int^T_0\sum_k \frac{d}{dt}\frac{\partial\hat{l}}{\partial{\eta_k}}\theta_k(X_t)\omega_k(U_t)+ \frac{\partial\hat{l}}{\partial{\eta_k}}\theta_k(\operatorname{ad}_{Y_t}X_t)\omega_k(\operatorname{ad}_{V_t}U_t)dt\\\nonumber
&=&\int^T_0\sum_k \frac{d}{dt}\frac{\partial\hat{l}}{\partial{\eta_k}}(\theta_k\oplus \omega_k)(X_t\oplus U_t)+ \operatorname{ad}^*_{Y_t}\left(\frac{\partial\hat{l}}{\partial{\eta_k}}\right)(\theta_k)(X_t)\,\operatorname{ad}^*_{V_t}(\omega_k)(U_k)dt\\ \nonumber
&=&\int^T_0\sum_k \frac{d}{dt}\frac{\partial\hat{l}}{\partial{\eta_k}}(\theta_k\oplus \omega_k)(X_t\oplus U_t)+ \frac{\partial\hat{l}}{\partial{\eta_k}}\left( \operatorname{ad}^*_{Y_t}(\theta_k)\oplus \operatorname{ad}^*_{V_t}(\omega_k)\right)(X_t\oplus U_t)dt
\end{eqnarray}
So, by the Lagrange-D'Alembert principle 
\begin{eqnarray}{\label{EDI2}}
\delta \int^T_0L(g, \dot{g})dt= 0 \;\; \text{ if and only if } \;\; \sum_k \frac{d}{dt}\frac{\partial\hat{l}}{\partial{\eta_k}}(\theta_k\oplus \omega_k)+ \frac{\partial\hat{l}}{\partial{\eta_k}}\left( \operatorname{ad}^*_{Y_t}\theta_k\oplus \operatorname{ad}^*_{V_t}\omega_k\right)=0
\end{eqnarray}
or 
\begin{eqnarray}{\label{EDI3}}
 \frac{d}{dt}\frac{\partial l}{\partial{Z_t}} + \left( \operatorname{ad}^*_{Y_t}\oplus \operatorname{ad}^*_{V_t}\right) \frac{\partial l}{\partial{Z_t}} = 0 \;\; , \;\; Z_t = Y_t\oplus V_t \in A\Upsilon
\end{eqnarray}
The last equation is the Euler--Poincar\'e equation on the trivial Lie groupoid.

\section*{Generalized Rigid Body on the Sphere}

For the trivial Lie groupoid $S^2 \times SO(3) \times S^2$, the reduced variables are
\(\eta(x,t)\in\mathfrak{so}(3)\) and \(X(x,t)\in T_x S^2\).
A natural reduced Lagrangian, independent of position, is
\begin{equation}
\ell(\eta,X) \;=\; \tfrac12 \int_{S^2} \Big( \langle I \eta(x), \eta(x)\rangle \;+\; \alpha |X(x)|^2 \Big)\, dS(x),
\end{equation}
with momenta
\begin{equation}
\mu \;=\; \frac{\delta \ell}{\delta \eta} = I\eta,
\qquad
m \;=\; \frac{\delta \ell}{\delta X} = \alpha X^\flat.
\end{equation}

\section*{Euler--Poincar\'e Equations (Position-Independent Case)}
Since the Lagrangian does not explicitly depend on $x \in S^2$, the EP equations reduce to
\begin{equation}
\partial_t(\mu \oplus m)
+ \big( \operatorname{ad}^*_\eta \oplus \operatorname{ad}^*_X \big)(\mu \oplus m) \;=\; 0.
\end{equation}

In component form, this is
\begin{align}
\partial_t \mu \;+\; \mathcal{L}_X \mu \;+\; \mu \times \eta &= 0, \\[6pt]
\partial_t m \;+\; \mathcal{L}_X m \;+\; \mu \diamond \eta &= 0.
\end{align}

\section*{Remarks}
\begin{itemize}
  \item $\mathcal L_X$ denotes the Lie derivative along the vector field $X$ on $S^2$.
  \item $\mu \times \eta$ is the coadjoint action in $\mathfrak{so}(3)^* \cong \mathbb R^3$.
  \item The diamond operator is defined by
  \[
  \langle \mu \diamond \eta, Y \rangle = -\int_{S^2} \langle \mu, \mathcal L_Y \eta \rangle\, dS,
  \quad \forall Y \in \mathfrak X(S^2).
  \]
\end{itemize}

\section{Application: Collective Cell Migration on $S^2$}

We illustrate a biological application of the Euler--Poincar\'e (EP) equation formulated on the trivial Lie groupoid $S^2 \times SO(3) \times S^2$. The base manifold $S^2$ represents a spherical tissue, while the Lie group $SO(3)$ models the internal polarity of migrating cells (e.g.\ actin orientation). The velocity field $X$ describes motion of cells on the surface, and the group component $\eta$ encodes internal polarity dynamics.

\section*{Reduced Lagrangian and EP equations}
The reduced Lagrangian is
\begin{equation}
  \ell(\eta,X) = \tfrac12 \int_{S^2} \big( \langle I\eta,\eta\rangle + \alpha |X|^2 \big)\, dS,
\end{equation}
where the first term penalizes the energetic cost of polarity rearrangement and the second term penalizes mechanical locomotion. 
Let the conjugate momenta be $\mu = \delta\ell/\delta\eta$ and $m=\delta\ell/\delta X$. 
The EP equations on the semi-direct product read, in compact block form,
\begin{equation}
  \partial_t(\mu \oplus m) + \big(\operatorname{ad}^*_{\eta} \oplus \operatorname{ad}^*_{X}\big)(\mu \oplus m) = 0,
\end{equation}
i.e.
\begin{align}
\partial_t \mu + \mathcal{L}_X \mu + \mu \times \eta &= 0, \\
\partial_t m + \mathcal{L}_X m + \mu \diamond \eta &= 0.
\end{align}
Here $\mathcal{L}_X$ denotes the Lie derivative along $X$, and $(\mu\diamond\eta)\in\Omega^1(S^2)$ is defined by
$\langle \mu\diamond\eta, Y\rangle = -\int_{S^2}\langle \mu, \mathcal{L}_Y \eta\rangle\, dS$ for all $Y\in\mathfrak X(S^2)$.

\section*{Biological parameters and optimality}
For a spherical embryo of radius $R=100\,\mu m$ and an arc length $L \approx 105\,\mu m$ (geodesic of $60^\circ$), choosing realistic parameters $\alpha = 10^{-16}\, \mathrm{J}\cdot \mathrm{min}^2/\mu m^2$ and $\lambda = 10^{-17}\, \mathrm{J}/\mathrm{min}$ leads to the energy--time tradeoff
\begin{equation}
  J_\lambda(T) = \tfrac12 \alpha \frac{L^2}{T} + \lambda T,
\end{equation}
whose minimizer is $T^*=\sqrt{ \tfrac{1}{2}\alpha L^2/\lambda }\approx 74\,\mathrm{min}$, with average velocity $v^*\approx 1.4\,\mu m/\mathrm{min}$ (see Figure \ref{fig:num-ep-combined}).

\begin{figure}[h]
  \centering
  \includegraphics[width=0.8\textwidth]{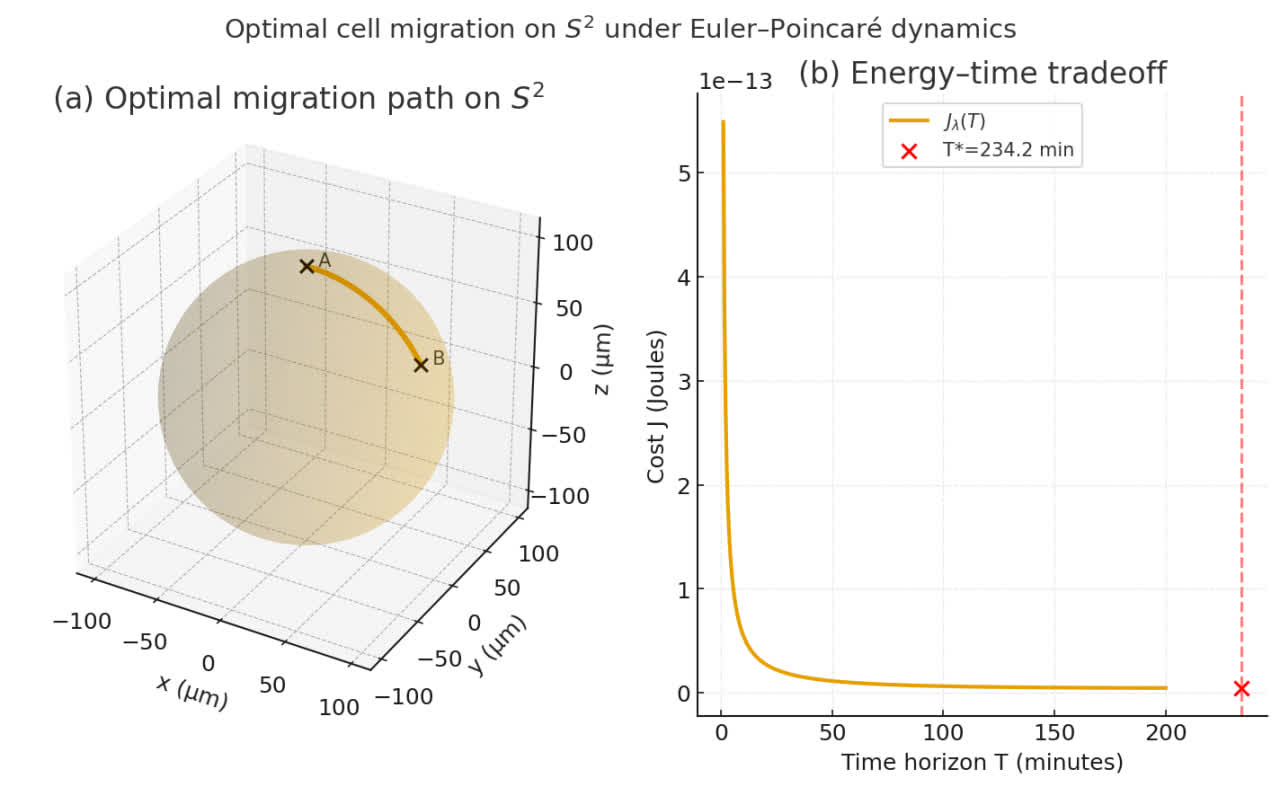}
  \caption{(a) Optimal geodesic trajectory on $S^2$. (b) Cost $J_\lambda(T)$ with unique optimum at $T^*$.}
  \label{fig:num-ep-combined}
\end{figure}

\section*{Comparison Between Lie Group and Trivial Lie Groupoid}

\renewcommand{\arraystretch}{1.4}
\setlength{\tabcolsep}{8pt}
\begin{tabular}{|p{3cm}|p{5cm}|p{7cm}|}
\hline
\textbf{Aspect} & \textbf{Lie Group $SO(3)$} & \textbf{Trivial Lie Groupoid $S^2 \times SO(3) \times S^2$} \\
\hline
\textbf{State variables} & 
$\eta(t)\in \mathfrak{so}(3)$ (angular velocity / polarity rotation) &
$(\eta(t), X(x,t))$, with $X\in \mathfrak{X}(S^2)$ (internal rotation + spatial velocity field) \\
\hline
\textbf{Momenta} &
$\mu = \delta \ell/\delta \eta \in \mathfrak{so}(3)^*$ &
$\mu = \delta \ell/\delta \eta,\; m = \delta \ell/\delta X$ \\
\hline
\textbf{EP equations} &
$\dot\mu + \operatorname{ad}^*_\eta \mu = 0$ &
$\dot\mu + \mathcal L_X\mu + \mu\times\eta=0,$ \\
& &
$\dot m + \mathcal L_X m + \mu\diamond\eta=0$ \\
\hline
\textbf{Geometry} &
Dynamics constrained to the Lie group $SO(3)$ &
Dynamics on the semi-direct product $\mathfrak{so}(3)\oplus TS^2$ \\
\hline
\textbf{Biological meaning} &
Cell \textbf{rotates its polarity only}, no displacement on tissue &
Cell undergoes \textbf{both polarity rotation} and \textbf{spatial migration} on the spherical tissue \\
\hline
\textbf{Optimal control cost} &
$J(T)=\int_0^T \tfrac12\langle I\eta,\eta\rangle dt + \lambda T$ 
$\;\;\rightarrow$ minimal internal rearrangement &
$J(T)=\int_0^T\int_{S^2} \tfrac12(\langle I\eta,\eta\rangle+\alpha |X|^2)\,dS\,dt+\lambda T$ 
$\;\;\rightarrow$ minimal internal + locomotion energy \\
\hline
\textbf{Optimal solution} &
Geodesic in $SO(3)$: uniform rotation with constant $\eta$ &
Coupled geodesic migration on $S^2$ plus rotation; both internal and spatial trajectories optimized \\
\hline
\end{tabular}

\section*{Reduced Model (Analytical/Numerical Illustration)}
We consider the trivial Lie groupoid $S^2 \times SO(3) \times S^2$ with the reduced Lagrangian
\begin{equation}
\ell(\eta,X) = \tfrac12 \int_{S^2} \big(\langle I\eta,\eta\rangle + \alpha |X|^2\big)\, dS,
\end{equation}
and specialize to geodesic paths on $S^2$ with constant speed. For fixed endpoints $A,B\in S^2$ at geodesic distance $L$ (here $L \approx 105\,\mu m$), the energy--time tradeoff
\begin{equation}
J_\lambda(T) = \tfrac12 \alpha \frac{L^2}{T} + \lambda T
\end{equation}
admits a closed-form optimum $T^* = \sqrt{\tfrac{1}{2}\alpha L^2/\lambda}$. With
$\alpha = 10^{-16}\,\mathrm{J}\cdot \mathrm{min}^2/\mu m^2$ and $\lambda = 10^{-17}\,\mathrm{J}/\mathrm{min}$
we obtain $T^* \approx 74\,\mathrm{min}$ and $v^* = L/T^* \approx 1.4\,\mu m/\mathrm{min}$.

\section*{Full Field Model on $S^2$ (EP PDE + Optimal Control)}
We solve the EP system for $(\mu,m)$ coupled to $(\eta,X)$ on $S^2$:
\begin{align}
\partial_t \mu + \mathcal L_X \mu + \mu \times \eta &= 0,\\
\partial_t m + \mathcal L_X m + \mu \diamond \eta &= 0,
\end{align}
with the simple constitutive choice $m=\alpha X^\flat$. We minimize the space--time cost
\begin{equation}
\mathcal J[\eta,X] = \int_0^T\!\!\int_{S^2} \tfrac12 \big(\langle I\eta,\eta\rangle + \alpha|X|^2\big)\, dS\, dt
\; + \; \lambda T,
\end{equation}
subject to EP constraints and boundary conditions (e.g.\ advection of tracers from $A$ to $B$).

\section{Remark on Numerical Aspects}
In this article our focus is on the theoretical framework and applications of the Euler--Poincar\'e equations on Lie groupoids. While numerical discretization and optimization schemes are important for practical simulations, we have not developed them here in detail. Such computational aspects — including mesh discretizations on $S^2$, time-stepping methods, and optimization algorithms — are left for future research.

\section{Conclusion}

In this work, we extended the Euler--Poincar\'e framework from Lie groups to trivial Lie groupoids and demonstrated its potential through both theoretical developments and illustrative examples. A particularly insightful application arises in the context of collective cell migration on spherical tissues. The EP equations capture the coupled dynamics of spatial migration and internal polarity rotation, providing a unified perspective that is not accessible when restricting to Lie groups alone.

Future research may explore numerical schemes adapted to groupoid-based dynamics, and investigate further applications ranging from morphogenetic processes to swarm robotics.

\small{

}


\begin{thebibliography}{99}
\bibitem{BaWi}Baillieul J.C. Willems, \textit{Mathematical Control Theory}, Originally published by Springer Verlag New York, Inc. in 1999.
\bibitem{MarRat} Marsden, J.E., Ratiu, T.S. (1999). Introduction to Mechanics and Symmetry: A Basic Exposition of Classical Mechanical Systems. Springer.
\bibitem{HolMarRat}  Holm, D.D., Marsden, J.E.,  Ratiu, T.S. \textit{The Euler–Poincaré equations and semidirect products with applications to continuum theories. } Advances in Mathematics, 137(1), (1998) 1–81.
\bibitem{Bloch}Bloch, A.M., \textit{ Nonholonomic Mechanics and Control}. 2nd Edition, Springer,(2015).
\bibitem{Bos} R. Bos, \textit{Geometric quantization of Hamiltonian actions of Lie algebroids and Lie groupoids}, Int. J. Geom. Methods Mod. Phys, 4 (2007), 389-436.
\bibitem{CaDr} Cabrera, A.,  Drummond, T.  \textit{Lie groupoids and Lie algebroids in classical mechanics}. SIGMA, 11, 014. (2015).
\bibitem{Wei} Weinstein, A.  \textit{Groupoids: unifying internal and external symmetry}. Notices of the AMS, 43(7), (1996) 744–752.
\bibitem{LeMarMar} de León, M., Marrero, J.C.,  Martínez, E. \textit{ Lagrangian submanifolds and dynamics on Lie algebroids}. Journal of Physics A: Mathematical and Theoretical, 40(32), (2007) 10079–10106.
\bibitem{BuLe} Bullo, F.,  Lewis, A.D. \textit{ Geometric Control of Mechanical Systems}: Modeling, Analysis, and Design for Simple Mechanical Control Systems. Springer, (2004).
\bibitem{Hagh} Haghighatdoost, G. Optimal control problems on the co-adjoint Lie groupoids. Archives of Control Sciences, 34(4), (2024) 715–742.
\bibitem{HA2} Gh. Haghighatdoost and R. Ayoubi, \textit{Hamiltonian systems on co-adjoint Lie groupoids}, Journal of Lie Theory 31, 2 (2021), 493-516. 
\bibitem{Hagh2} Gh. Haghighatdoost and F. Hasani, \textit{Lie group and Lie algebra 1}, Payame noor university, 1395.
\bibitem{Mack} K.C.H. Mackenzie, \textit{General theory of Lie groupoids and Lie algebroids}, London Math. Soc. Lecture notes series 213, Cambridge University Press, Cambridge, (2005).
\bibitem{Marle} C.M. Marle, \textit{Lie, symplectic and Poisson groupoids and their Lie algebroids}, arXiv preprint arXiv:1402.0059, (2014).
\bibitem{GrGl}F. Graner and J.A. Glazier, \textit{Simulation of biological cell sorting using a two-dimensional extended Potts model}, Physical Review Letters, 69(13), (1992) 2013–2016.
\bibitem{ViZa} T. Vicsek and A. Zafeiris, \textit{Collective motion}, Physics Reports, 517(3–4), (2012) 71–140.
\bibitem{Ed} L. Edelstein-Keshet, \textit{Mathematical Models in Biology}, SIAM, 2005.
\bibitem{DrHo} D. Drasdo and S. Höhme, \textit{Individual-based approaches to birth and death in avascular tumors}, Mathematical and Computer Modelling, 37(11), (2003) 1163–1175.
\bibitem{AlTr}R. Alert and X. Trepat, \textit{Physical models of collective cell migration}, Annual Review of Condensed Matter Physics, 11, (2020) 77–101.
\bibitem{Schmeding} A. Schmeding, \textit{The Lie group of vertical bisections of a regular Lie groupoid}, Forum Mathematicum, 32, 479--489, (2019).
\end{thebibliography}
\end{document}